\documentclass[12pt]{amsart}
\usepackage{amssymb,latexsym,amscd}
\let\eqref\relax
\usepackage{amsmath}
\usepackage{amsfonts}
\def\eqref#1{(\ref{#1})}   
\def\Q{\mathbb Q}
\def\C{\mathbb C}

\def\Z{\mathbb Z}

\def\PSL{\mathop{\mathrm{PSL}}\nolimits}
\def\PGL{\mathop{\mathrm{PGL}}\nolimits}

\def\Aut{\mathop{\mathrm{Aut}}\nolimits}
\def\Stab{\mathop{\mathrm{Stab}}\nolimits}
\def\Core{\mathop{\mathrm{Core}}\nolimits}
\def\Gal{\mathop{\mathrm{Gal}}\nolimits}
\def\Id{\mathop{\mathrm{Id}}\nolimits}

\def\F{\mathbb{F}}
\def\O{\mathfrak{O}}
\def\Sym#1{\mathbf{S}_{#1}}
\def\Alt#1{\mathbf{A}_{#1}}
\def\pitem{\advance\labelsep3mm\item \advance\leftskip3mm\advance\linewidth-3mm}
\def\mitem{\advance\labelsep-3mm\item \advance\leftskip-3mm\advance\linewidth3mm}
\newtheorem{definition}{Definition}[section]
\newtheorem{corollary}[definition]{Corollary}
\newtheorem{lemma}[definition]{Lemma}

\newtheorem{remark}[definition]{Remark}
\newtheorem{theorem}[definition]{Theorem}
\newtheorem{example}[definition]{Example}
\newtheorem{algorithm}[definition]{Algorithm}
\title{Computation of Galois groups of rational polynomials}
\author{Claus Fieker}
\address{Fachbereich Mathematik, Universit\"at Kaiserslautern, Postfach 3049, 67653 Kaiserslautern}
\email{fieker@mathematik.uni-kl.de}

\author{J\"urgen Kl\"uners}
\address{Mathematisches Institut
der Universit\"at Paderborn, Warburger Str. 100, 33098 Paderborn, 
Germany}
\email{klueners@math.uni-paderborn.de}

\begin{document}
\maketitle
\begin{abstract}
  Computational Galois theory, in particular the problem of computing
  the Galois group of a given polynomial is a very old problem.
  Currently, the best algorithmic solution is Stauduhar's method.
  Computationally, one of the key challenges in the application of
  Stauduhar's method is to find, for a given pair of groups $H<G$ a
  $G$-relative $H$-invariant, that is a multivariate polynomial $F$
  that is $H$-invariant, but not $G$-invariant.  While generic,
  theoretical methods are known to find such $F$, in general they
  yield impractical answers. We give a general method for computing
  invariants of large degree which improves on previous known methods,
  as well as various special invariants that are derived from the
  structure of the groups.  We then apply our new invariants to the
  task of computing the Galois groups of polynomials over the rational
  numbers, resulting in the first practical degree independent
  algorithm.
\end{abstract}

\section{Introduction}
Computational Galois theory, in particular the problem of
finding the Galois group of a given polynomial is a very old problem.
While various algorithms have been published, so far they are either
impractical for groups of size $>1000$ due to the requirement of 
exact representation of an algebraic splitting field, or they are
degree dependent.
Algorithms of the first kind include for example the 
naive approach of constructing a splitting field by repeated
factorisation as well as more sophisticated methods \cite{yokoyama}.
Algorithms of the second kind fall broadly into two approaches:
a classical approach that aims to characterise the Galois group as an
abstract group by building a decision tree using certain indicators
(resolvent polynomials) \cite[Chapter 6.3]{Cohen1}
and a newer approach, by Stauduhar \cite{Stauduhar} where the Galois group is
constructed explicitly as a group of permutations of the roots of the
polynomial in question. Stauduhar's method roughly works by
traversing the lattice of (transitive) subgroups of the full
symmetric group from the top ($\Sym n$) down to the Galois
group of the polynomial. At each step, this is done through the
help of invariants and the high precision evaluation of those.

This paper naturally splits into two parts: the first discussing the
problem of finding a useful invariant for each pair of groups (see
Section \ref{Not} for a precise statement), and the second part
explaining how this is used to compute Galois groups of polynomials
over $\Q$, see Section \ref{sec7} for details.

Primitive invariants for permutation groups, i.e. multivariate
polynomials with a given stabiliser, are among the most important
objects in computational Galois theory. They are the central
ingredient in Stauduhar's method \cite{Gei2,GeKl} for the
determination of the Galois group of a polynomial $f$: given two
groups $H<G$ a ($G$-relative) $H$-invariant is used to decide if
$\Gal(f)\leq H^g$ for some $g\in G$ under the assumption that
$\Gal(f)\leq G$.  Furthermore, applications, such as the explicit
realization of Galois groups by explicitly computing defining
equations for subfields of the splitting field for $f$ rely on
invariants as well \cite{KlMa}.

While there are a few methods known for the computation of such invariants
in the literature, in applications, invariants were mostly
the result of ad-hoc methods. Generic algorithms, eg. 
\cite{TheFrenchGirl,girstmair}
for individual invariants or \cite{KemperSteel} for the computation
of the entire ring of invariants become rapidly unpractical for
larger degree permutation groups. 

It should be stressed that while invariant theory gives explicit invariants
for all pairs of groups $H<G$, the generic results tend to be 
impractical as the resulting invariants are computationally far too
complex.

In what follows, we will give a new, space-efficient algorithm to
compute all invariants of a given degree for arbitrary pairs of
groups, and for maximal subgroups of transitive groups we give several
constructions that allow the determination of efficient invariants in
many cases. We then demonstrate in Section \ref{special} that knowledge of
the subgroup structure can also be used to find efficient invariants,
as frequently invariants for some subgroups can be combined to give
invariants for others.

Finally, we demonstrate the efficiency and the limits of our methods by
considering several examples.

\section{Notation}\label{Not}
Transitive groups of degree $<32$ are denoted by $nTm$ where $n$ is the degree
and $m$ is the number of the group in the classification \cite{hulpke} used
by both Magma and Gap.
For the rest of the article, we fix some positive integer $n$.
The symmetric group on $n$ elements, $\Sym n$ acts on the polynomial
ring $\Z[\underline X] = \Z[X_1, \ldots, X_n]$ in $n$ variables via
$$X_i \mapsto X_{\sigma(i)}.$$
For $\sigma\in \Sym n$ we usually write $F^\sigma$ for the image under
this map.
A polynomial $F\in \Z[\underline X]$ is called a $H$-invariant (for some
group $H\leq \Sym n$) if $F^\sigma = F$ for all $\sigma \in H$.
Given two subgroups $H<G\leq \Sym n$, we call a polynomial 
$F\in \Z[\underline X]$ a $G$-relative $H$-invariant, if its 
stabiliser 
$\Stab_G F := \{\sigma \in G \mid F^\sigma = F\}$ in $G$ equals $H$.
A polynomial $F\in \Z[\underline X]$ is called an absolute $H$-invariant
if $\Stab_{\Sym n} F = H$.

For any subgroup $H\leq \Sym n$ we can consider the ring $\Z[\underline
X]^H$ of absolute $H$-invariants and also the invariant field
$\Q(\underline X)^H$ of rational functions that are invariant under
$H$.
\begin{remark}
If $H<G\leq \Sym n$ is a pair of subgroups and if $F\in \Z[\underline X]$
is a $G$-relative $H$-invariant, then
\begin{enumerate}
\item As an extension of fields, $\Q(\underline X)^H$ is a finite
  extension of $\Q(\underline X)^G$ of degree 
  $$[\Q(\underline X)^H : \Q(\underline X)^G] = (G:H)$$
\item Furthermore 
  $$\Q(\underline X)^H = \Q(\underline X)^G[F]$$
  that is, $F$ is a primitive element for the extension.
\item From the main theorem on symmetric functions it follows that
  $$\Z[\underline X]^{\Sym n} = \Z[s_1, \ldots, s_n]$$
  where $s_i = \sum\limits_{1\leq j_1<\cdots <j_i\leq n}
  \prod\limits_{\ell=1}^i X_{j_\ell}$ are the elementary symmetric
  functions.
\end{enumerate}
\end{remark}

%
%

\section{Stauduhar's method}

In this section we recall the necessary tools from Stauduhar's method.
We do this in a slightly more general context which has the advantage that we 
can combine the
information obtained by the resolvent method and by Stauduhar's method.

Let us assume that we are given a monic polynomial $f\in\Z[X]$ of
degree $n$ and we would like to compute the Galois group of $f$.
Certainly, the Galois group is a subgroup of $\Sym n$ acrting on the roots of $f$ and therefore we
can assume that we know a subgroup $G\leq \Sym n$ with $\Gal(f)\leq
G$. Assume furthermore that we have a proper subgroup $H<G$ and let
$F\in \Z[X_1, \ldots, X_n]$ be a $G$-relative $H$-invariant
polynomial. In the following we denote by $G//H$ a set of
representatives of right cosets $H \sigma$ of $G/H$. We remark that we only use right cosets in this paper. 
The following is proved in
\cite{Stauduhar}.
\begin{lemma}\label{specialStaud}
  Let $F$ be $G$-relative $H$-invariant and assume that $\Gal(f)\leq G$, where $\Gal(f)$ acts on
  the roots $\alpha_1,\ldots,\alpha_n$ in some fixed closure. Then
  $$R_F := \prod_{\sigma\in G//H} (T-F^\sigma(\alpha_1, \ldots, \alpha_n)) \in \Z[T].$$
  $R_F$ is called the relative resolvent polynomial (corresponding to $H<G$ and $F$).
\end{lemma}
\begin{proof}
  Since $F^H=F$ we see that $R_F$ does not depend on the choice
  of coset representatives. The polynomial $R_F$ is invariant
  under $G$ and since $\Gal(f)\leq G$ it is invariant under $\Gal(f)$. Therefore
  all coefficients of $R_F$ are in $\Q$ and also algebraic integers,
  thus in $\Z$.
\end{proof}
Suppose that $R_F$ is squarefree and we know a non-trivial factor of
$R_F$ in $\Z[T]$. In this situation we show in the following theorem that
the Galois group of $f$ is contained in a proper subgroup of $G$ and
therefore we make progress. In case $R_F$ is not squarefree, we apply
a Tschirnhausen transformation $t\in \Z[x]$ and compute a new
polynomial
$$R_{F, t} := \prod_{\sigma\in G//H} (T-F^\sigma(t(\alpha_1), \ldots, t(\alpha_n))).$$
It can be shown \cite{girstmair2} that there exist suitable
transformations $t$ such that $R_{F, t}$ is squarefree. Furthermore,
introducing $t$ amounts to a change of $f$ that will not affect the
Galois group.

\begin{theorem}\label{genStaud}
  In the situation of Lemma \ref{specialStaud}, assume that $R_F$ is squarefree and $A \in
  \Z[T]$ is a divisor of $R_F$, of degree $\deg A =m$.  Denote by
  $\rho: G \rightarrow \Sym{G/H}$ the permutation action on right
  cosets $G/H$. Then there exist $\sigma_1,\ldots,\sigma_m\in G$ such
  that
  $$A(T) = \prod_{i=1}^m \left(T-F^{\sigma_i}(\alpha_1,\ldots,\alpha_n)\right).$$
  Denote by $B$ the set of right cosets $\{H\sigma_i \mid 1\leq i \leq m\}$. 

  Then $\Gal(f)\leq \rho^{-1} \left(\Stab_{\rho(G)} (B)\right)$.
\end{theorem}
\begin{proof}
  The elements $\sigma_1,\ldots,\sigma_m$ are in pairwise different right cosets of
  $G/H$ since otherwise $F^{\sigma_i} = F^{\sigma_j}$
  and the polynomial $A$ is not squarefree. Extend the $\sigma_i$ to a
  complete system of representatives $\sigma_1,\ldots, \sigma_r$ of
  $G/H$, where $r=(G:H)$. Now let $\tau\in \Gal(f)\leq G$ be an arbitrary
  element. The elements $\tau\sigma_1,\ldots,\tau\sigma_r$ are also a set
  of representatives of $G/H$. Since $A$ is invariant under $\tau\in\Gal(f)$,
  for $1\leq i\leq m$ we have  $\tau\sigma_i \in H\sigma_j$ with $j\leq m$.
  Therefore we get that $\rho(\tau) \in \Stab_{\rho(G)} (B)$. 
\end{proof}
The case of linear factors in Theorem \ref{genStaud} was already
proved in \cite{Stauduhar}, in fact it formed the key technique in the
original paper. The possible use of quadratic factors is mentioned on
the last page of \cite{soicher} but is rejected there since the
practical group theory would have been too complicated. The general statement
is also proven in \cite[Satz 2.4]{Gei2}, although only the case of
linear factors is used to determine groups. Higher degree factors are
only considered in a verification step.  

We can also apply this theorem if we know more than one factor.
\begin{corollary}
  Assume that $R_F$ is squafree and factors as $R_F=A_1\cdots A_s$ with $A_i\in \Z[T]$.
  Denote by $B_i$ the set of right cosets of $G/H$ corresponding to $A_i$.
  Then 
  $$\Gal(f) \leq \bigcap_{i=1}^s \rho^{-1}(\Stab_{\rho(G)} B_i).$$ 
\end{corollary}
Because of its importance we describe the case of linear factors in more detail in the following corollary.
\begin{corollary}\label{StaudLin}
  Assume that $R_F$ is squarefree and has a linear factor in $\Z[T]$ corresponding to
  $H\sigma$ for $\sigma\in G$. Then $\Gal(f)\leq H^\sigma:=\sigma^{-1}H\sigma$.
\end{corollary}
\begin{proof}
  Note that a point stabiliser of $\rho(G)$ is isomorphic to $H$.
\end{proof}
We remark that in the following we mostly use this Corollary since
finding linear factors is much easier than doing a complete
factorisation.  In particular, we are frequently able to find linear
factors without ever constructing $R_F$ completely. We remark that the 
complexity of this method depends on the index $(G:U)$. Even, if
we do not compute the corresponding resolvent polynomial of degree
$(G:U)$ directly, the coefficient bounds that we need in our algorithm
are dependent on this index, too.

If the index $(G:U)$ is huge it could be nice to work with a subgroup
$H$ of smaller index and use higher degree factors in order to prove
that $\Gal(f)$ is contained in a conjugate of $U$.

\begin{example}\label{largefactor}
  Let $f\in\Z[X]$ be a polynomial with Galois group 19T5 $\cong C_{19}
  \rtimes C_9$. This is a maximal subgroup of $\Alt{19}$ of index
  $17!$. Since $19T5$ is not 2-transitive, take
  $S:=\Stab_{\Alt{19}} ([1,2])$ the intersection of the point stabilisers of 1
  and 2, take $F := x_1-x_2$ and compute the resolvent $R_F$.  This is a
  polynomial of degree $19\cdot 18$ which has $\alpha_i-\alpha_j$ as roots
  for $1\leq i\ne j\leq 19$ (assuming that those roots are different). 
  Furthermore, using resultants, $R_F$ can be computed symbolically without
  explicit knowledge of $S$ or $F$.
  Factorisation of $R_F$ finds
  two factors of degree 171. When we apply Theorem \ref{genStaud} we directly
  descend to the correct Galois group.
\end{example} 
Let $\alpha$ be a root of the polynomial $f$ in the last example. The
factorisation approach of the last example is equivalent to the fact
that $f/(x-\alpha) \in\Q(\alpha)[x]$ factorises into two degree 9
factors.

Very often the factorisation approach is not optimal or even feasible
since the
degree of the resolvent polynomial is too high for efficient
factorisation.
In some situations our algorithm produces a Galois group (as
an actual permutation group on the roots) which is only correct with a very
high probability. In this situation we can turn to the factorisation
method in order to check our result, i.e. to give a proof that the
result is mathematically correct. Since we assume knowledge of the action on the
roots, it is not necessary to factor the resolvent polynomial. By
analysing the proof of Theorem \ref{genStaud} we can write down the
factor and check if it is in $\Z[X]$ and whether it divides the resolvent
polynomial, see Section \ref{check} for more details.
 Similar ideas were used by Casperson and McKay \cite{McKay} to 
obtain polynomials with Galois group $M_{11}$. 
\begin{example}\label{largefactor2}
  Let $p$ be a prime number, $G:=\Sym{p+1}$, and $H:=\PGL(2, p)\leq
  G$. Then $(G:H) = (p-2)!$ and $H$ is a maximal subgroup. Furthermore
  $G$ is sharply 3--transitive which means that for the resolvent
  method we have to use a polynomial acting on 4-sets of the roots
  which has degree $\binom{p+1}{4}$.  For p=19 this polynomial has
  degree 4845 and splits into four factors of degree 570, 855, 1710,
  and 1710 resp. In our implementation we compute the Galois group
  using short cosets (see Remark \ref{short-coset}) with a very high
  probability. By applying the methods of Section \ref{check} we use
  this group to approximate the degree 570 factor. When computing the
  corresponding stabiliser according to Theorem \ref{genStaud} we
  descend to $G$.
\end{example}
We remark that the group $\PSL(2,p)\leq \Alt{p+1}$ is not
3--transitive. The 3-set polynomial which is the resolvent corresponding
to $S:=\Stab_{\Alt{p+1}}(\{1,2,3\})$ gives no information since this
polynomial stays irreducible. But if we take
the pointwise stabiliser $S:=\Stab_{\Alt{p+1}} [1,2,3]$ (intersection
of the 3 pointwise stabilisers) then we will find 3 factors. The
latter polynomial has only degree $(p+1)p(p-1)$ compared to
$\binom{p+1}{4}$.

As a final example in this section we consider a degree 40 polynomial 
with Galois group PGsp(4,3). This polynomial was computed in \cite{DoDo}
and for reasons of space we do not give the actual
polynomial here. It comes from the 3-torsion of a hyperelliptic curve
  $$C: y^2 + (-x^2-1)y = x^5-x^4+x^3-x^2.$$ The Galois group is primitive and not 2--transitive.
Furthermore this group is maximal in $\Alt{40}$. The algorithm
outlined below, using only linear factors
computes the Galois group within 50 seconds. However, factoring
a suitable resolvent for the stabiliser of 2-sets, completes the
computation in only 20 seconds.

\section{Generic Invariants}\label{generic}
We fix two groups $H < G\leq \Sym n$ and assume unless explicitly
stated otherwise, that $H$ is a maximal subgroup of $G$. The aim of
this section is to find a $G$-relative $H$-invariant $F\in
\Z[\underline X]$ of small degree and a small number of terms. While
the first aim can be obtained easily, the second is more difficult,
and will be discussed later. To simplify notation we will write
$\sum A$ to mean $\sum_{a\in A} a$ for suitable sets $A$, usually
orbits.

The first observation is that it is always easy to write down some
invariant. Certainly, every $\Sym n$-relative invariant is $G$-relative as
well for $H\leq G\leq \Sym n$.
\begin{lemma}\label{genin}
$$F := \sum_{\sigma\in H}(\prod_{i=1}^{n-1} X_i^i)^\sigma$$
is a $\Sym n$-relative $H$-invariant.
\end{lemma}

While Lemma \ref{genin} proves the existence of $G$-relative $H$-invariants,
these are very expensive invariants from the point of view
of evaluation. Even assuming that the powers of the evaluation points 
are stored, the evaluation of each term needs $n-2$ multiplications, so
that in total $\#H(n-2)$ multiplications are necessary.
In order to improve on this we make use of the following well known 
facts (\cite{kemper}):
\begin{theorem}
For any polynomial $I\in \Z[\underline X]$, and every subgroup $H\leq \Sym n$ we have that
$F(\underline X):= \sum \{I^h(\underline X) \mid h\in H\}
 =: \sum I^H(\underline X)$ is $H$-invariant.

For every $H$-invariant polynomial $F\in \Q[\underline X]$,
there exist monomials
$m_i$ and coefficients $a_i\in \Q$ such that 
$$F = \sum_{i=1}^r a_i\sum m_i^H.$$
Thus invariants of the form $\sum m^H$ form a vector space
basis for the ring of all invariants.

The invariant ring $\Q[\underline X]^H$ of $H$-invariants
is a graded $\Q$-vector space. The dimensions
of the summands can be read of the Hilbert-series:
$$f_H(t) := \sum_{i=0}^\infty t^i \dim (R_H)_i$$
where $(R_H)_i = \{ r\in \Q[\underline X]^H \mid \deg r = i\} \cup \{0\}$.

  The Hilbert series can be computed from the knowledge of the set of the conjugacy classes $C$
  of $H$:
$$f_H(t) = \frac 1{\#H} \sum_{c\in C} \frac{\#c}{\prod_{i=1}^l(1-x^{c_i})^{d_i}},$$
where $(c_i, d_i)$ is the cycle structure of any representative of the class
$c$ of $H$.
\end{theorem}

To improve on Lemma \ref{genin} we will try to find a small invariant
as a basis element for some $(R_H)_d$ for $d$ as small as possible.
Unfortunately, there are pairs of groups, $G=\Sym{n}$, $H=\Alt n $ for example,
where the invariant in Lemma \ref{genin} can be shown to be of minimal
degree.

In the remainder of this section we will develop methods to compute
a basis for $(R_H)_d$ the vector space of $H$-invariant
polynomials of degree $d$ and also for the subspace of $G$-relative 
polynomials. Our strategy will be to first compute a basis for the
$\Sym{n}$-invariants and then show how to refine this basis.
We start with few observations:
\begin{remark}
\begin{enumerate}
\item
Let $F\in \Z[\underline X]$ be a polynomial and $H\leq \Sym n$ be a group. Then
$$ \sum_{\sigma\in H//\Stab_H(F)} F^\sigma = \sum F^H$$
and thus is $H$-invariant.
\item Let $m=\prod_{i=1}^nX_i^{a_i}$ be a monomial. Then we have
$$\Stab_H(m) = \bigcap_{a\in\{a_i \mid 1\leq i\leq n\}} \Stab_H(\{i \mid a_i=a\}),$$
thus stabilisers of monomials can be computed as intersections of 
stabilisers of points or sets. Of course, for $H=\Sym{n}$ those stabilisers
can be made explicit as direct products of suitable $\Sym{m}$ for $m<n$.
\item Let $\{1, \ldots, n\} = \cup_{i=1}^r A_i$ be a partition. Then
 $$\Stab_{\Sym{n}}(A_1, \ldots, A_r) \cong \prod_{i=1}^r \Sym{A_i}.$$
\end{enumerate}
\end{remark}

\subsection{$\Sym{n}$-invariants}
In this section we will develop methods to compute a basis for the
$\Sym{n}$-invariants as well as indicate how to improve on
the general method if we want to aim for relative invariants only. The
algorithm presented here is similar to the ideas presented in
\cite{TheFrenchGirl,girstmair}.

The key idea here is that the orbit sum
$$\sum_{\sigma\in\Sym{n}//\Stab(f)} f^\sigma$$ 
does not depend on the representative $f$ of the full orbit $f^{\Sym{n}}$
and that the action of the group on some monomial $m$ 
only depends on the partition
of $\{1,\ldots, n\}$ induced by $m = \prod_{i=1}^n X_i^{a_i}$:
$$\{1,\ldots, n\} = \dot\bigcup_{a\in \{a_i\mid 1\leq i\leq n\}} \{i \mid a_i = a\}.$$
On the other hand, by giving a partition $\underline A := \{A_i \mid i\}$ of $\{1,\ldots, n\}$
and pairwise different integers $a_i\geq 0$ 
the orbit of
$m(\underline a, \underline A) := \prod_{i=1}^s\prod_{j\in A_i}X_j^{a_i}$
is uniquely defined by $\underline A$ already.
Thus to solve our problem of finding a basis for $(R_{\Sym{n}})_d$ we simply
need to find all partitions and exponents such that
$\sum_{i=1}^sa_i\#A_i = d$.
We summarise this in an algorithm:
\begin{algorithm}\label{deg_d}
Let $d$ be an integer. The algorithm produces 
a basis for $(R_{\Sym{n}})_d$.
\begin{enumerate}
\item Let $I := \{\}$.
\item Compute the set $P$ of all partitions of $d$ of length at most $n$.
\item For $p\in P$ do
\pitem Let $p=(p_1, \ldots, p_i)$. 
  Append $I$ by $\prod_{j=1}^i X_j^{p_j}$
\end{enumerate}
\end{algorithm}

However, since we are eventually only interested in finding minimal
degree invariants we introduce more reductions here. The operation 
of $\Sym{n}$ on $m(\underline a, \underline A)$ does only depend on
$\underline A$, so for a minimal degree invariant we can also stipulate that
$a_i+1 = a_{i+1}$ - otherwise the same behaviour can be obtained
with smaller exponents. Similarly, minimal examples will be such that
$\#A_1 \geq \#A_2\geq \ldots\geq \#A_s$.
As an example: the orbits of $X_1 X_2^3$ and of $X_1 X_2^2$ are 
essentially the same, namely the orbit of $\{\{1\},\{2\}\}$. Thus if we are
looking for examples of minimal degree, then $X_1 X_2^3$ need not be
considered.

\subsection{$H$-invariants}
Let $H$ be a subgroup of $\Sym{n}$ and $I$ be a set of monomials
generating different $\Sym{n}$-invariants via orbit sums. Here we address
the problem of refining $I$ to contain a (maximal) set of monomials 
generating $H$-invariants.

Let $m$ be a monomial and $S := \Stab_G(m)$ its stabiliser in some group 
$\Sym{n}\geq G\geq H$. We use the following theorem
\begin{theorem}
Let $G\geq H$ be groups, $m$ a monomial and $S = \Stab_G(m)$ its
stabiliser. Furthermore, let $S\setminus G/H$ be the double cosets of $G$
with respect to $S$ and $H$ and let $\{g_i\mid 1\leq i\leq r\}$
be a set of representatives
(i.e., $G = \cup_{i=1}^r Sg_iH$ and $Sg_iH = Sg_jH$ if and only if $i=j$).
Then $\{m^g_i\mid 1\leq i\leq r\}$ generate linearly independent
$H$-invariants.
\end{theorem}
\begin{proof}
The linear independence is a direct consequence from the fact that
the double coset decomposition induces a decomposition of
$m^G$ into pairwise disjoint $H$-orbits.
\end{proof}
Thus the computation of $H$-invariants is reduced to the computation
of $\Sym{n}$-invariants followed by a double coset decomposition.
While in general double coset decompositions are hard to compute, it
is feasible here. We make use of the ladder-technique of \cite{schmalz}:
usually to compute double cosets, one computes a coset decomposition
with respect to one group and lets the other group act on them, thus
the complexity depends on the size of the index of the larger group 
in $G$. This procedure is frequently helped by computing a descending chain
from $G=:S_0>\cdots>S_j = S$
down to one smaller group, $S$ for example. The action 
of $H$ on $S\setminus G$ can then be deducted from the action of 
$H$ on $S_{i+1}\setminus S_i$.
Unfortunately, it is hard and frequently impossible to find good subgroup
chains, that is chains with small indices. The new idea introduced in 
\cite{schmalz} is to use a ladder rather than a chain, i.e. to allow
up-ward steps as well as down-ward ones. In order to use this technique
we therefore have to construct a suitable ladder.
This will be achieved by the following procedure:
\begin{algorithm}\label{ladder}
Let $G$ be a permutation group acting on $\Omega$ and $A\subseteq \Omega$
be arbitrary.
This algorithm will compute a ladder $G_i$ such that
$G = G_0$, $G_r=\Stab_G (A)$ and
if $G_i< G_{i+1}$ then $\#G_{i+1}/\#G_i\leq \#\Omega$ and
$G_i > G_{i+1}$ with $\#G_i/\#G_{i+1} \leq \#\Omega$ otherwise.
\begin{enumerate}
\item Let $B := \{\}$, and $i := 1$, $G_0:=G$.
\item for $a\in A$ do
\pitem 
  add $a$ to $B$ and
  compute $G_i := \Stab_{G_{i-1}} \{a\}$.
\item if $B \ne \{a\}$ then 
  $G_{i+1} := \Stab_G B$ and set $i := i+2$.
  Otherwise set $i:=i+1$.
 \end{enumerate}
\end{algorithm}
\begin{proof}
Let $A := \{a_1, \ldots, a_n\}$. The
properties of the $G_i$ are direct consequences of the following facts:
\begin{enumerate}
\item We either have 
 $G_i = \Stab_{G_{i-1}} \{\{a_1, \ldots, a_s\}, \{a_{s+1}\}\}$
 (in which case $G_i < G_{i+1}$)
 or
\item we have an up-ward step and obtain
  $G_{i+1} = \Stab_G \{a_1, \ldots, a_{s+1}\}$. Note that
  $\Stab_{G_{i+1}}\{a_{s+1}\} = G_i$.
\item For $G = \Sym{n}$, we have $\#\Stab_G \{a_1, \ldots, a_s\} = s!(n-s)!$ 
and
$\#\Stab_G \{\{a_1, \ldots, a_s\}, \{a_{s+1}\}\} = s!1!(n-s-1)!$.
\item In general, $\Stab_G A = G \cap \Stab_{\Sym{n}} A$, and for
      any groups $V < U < \Sym{n}$ we have
      $(U \cap G : V \cap G) \leq (U:V)$, thus the bound on the
      indices follows.
\end{enumerate}
\end{proof}
For more general partitions, the Algorithm \ref{ladder} will be called
repeatedly:
\begin{algorithm}\label{main_ladder}
  Let $G$ be a permutation group acting on $\Omega$ and $A = \{A_1,
  \ldots, A_s\}$ a partition of $\Omega$.  This algorithm will compute
  a ladder $G_i$ such that $G = G_0$, $G_r=\Stab_G (A)$ and if $G_i<
  G_{i+1}$ then $\#G_{i+1}/\#G_i\leq \#\Omega$ and $G_i > G_{i+1}$
  with $\#G_i/\#G_{i+1} \leq \#\Omega$ otherwise.
\begin{enumerate}
\item Let $U := G$.
\item For $a\in A$ do
\pitem Compute a ladder from $U$ to $\Stab_U a$ using Algorithm \ref{ladder}
 and print it.
\item Let $U := \Stab_U a$.
\end{enumerate}
\end{algorithm}

Let $G\leq \Sym n$ be arbitrary and $H<G$ a maximal subgroup. In order to compute 
$G$-relative $H$-invariants, we now use one of the following algorithms:
\begin{algorithm}
Let $H<G$ be as above and $d>0$ be an integer.
This algorithm will find a basis for the space
of $G$-relative $H$-invariants of degree $d$.
\begin{enumerate}
\item Compute a basis $B$ for $(R_{\Sym n})_d$ using \ref{deg_d}.
\item For each $b\in B$ do
\pitem Compute the corresponding partition $A$
\item  Use \ref{main_ladder} to compute a ladder $L$ from $\Sym n$ to
  $\Stab_{\Sym n}(b)$ using the partition $A$.
\item Use $L$ to compute a set $C$ of double coset representatives for
  $\Stab_{\Sym n}b \setminus \Sym n / H$
\item For each $c\in C$ do
\pitem Compute the indices of the stabilisers $(H:\Stab_H b^c)$
  and $(G:\Stab_G b^c)$. If they
  differ then $b^c$ generates a $G$-relative $H$-invariant. In this case
  print $\sum_{h \in H//\Stab_H b^c} b^{ch}$
\end{enumerate}
\end{algorithm}
The correctness of the algorithm follows immediately from the above discussions.
We remark that if we want only one invariant rather than a basis, we can
use a probabilistic approach:
\begin{algorithm}
Let $H<G$ as above and $d>0$ be an integer such there exists an $G$-relative
$H$-invariant of degree $d$.
This algorithm will find one
$G$-relative $H$-invariant of degree $d$.
\begin{enumerate}
\item Compute a basis $B$ for $(R_{\Sym n})_d$ using \ref{deg_d}.
\item repeat
\pitem Compute a random element $\sigma\in \Sym n$
\item For each $b\in B$ check if $(G:\Stab_G b^\sigma)$ differs
      from $(H:\Stab_H b^\sigma)$. If so, print 
      $\sum_{h \in H//\Stab_H b^\sigma} b^{\sigma h}$ and terminate.
\end{enumerate}
\end{algorithm}
To find a (minimal) degree $d$ such that there exists an $G$-relative
$H$-invariant we simply compute the difference of the 
Molien series $f_H(t) - f_G(t) = \sum_{i=1}^\infty s_it_i$ and take $d$ 
as the index of any non zero coefficient.

\section{Special Invariants}\label{special}
Like in the previous section we assume that $H$ is a maximal subgroup of $G$.
We use the maximality in our proofs to show that an $H$-invariant $F$ is $G$-relative,
if there exists one element $g\in G\setminus H$ with $F^g\ne F$.

Unfortunately, there are examples where the generic invariants are 
too expensive to compute or the given presentation needs too many arithmetic
operations to evaluate the invariant. The best known example for this
are the groups $H=\Alt{n}$ and $G=\Sym n$. Clearly, $H$ is a maximal subgroup of 
$G$ and the invariant 
$$F_1(X_1,\ldots,X_n) := \sum_{\sigma\in H}(\prod_{i=1}^{n-1} X_i^i)^\sigma$$ 
given in Lemma \ref{genin} is a $\Sym n$-relative $\Alt n$-invariant polynomial
of smallest possible total degree. If we store the powers of $X_i$ we 
need $(n-2) n!/2$ multiplications in order to evaluate this invariant.
If the characteristic is not equal to 2, then a better invariant is well known:
(for any $\sigma\notin H=\Alt n$) 
$$F_2(X_1,\ldots,X_n):=\prod_{1\leq i<j\leq n} (X_i-X_j) = F_1 - F_1^\sigma,$$
which can be evaluated using $n(n-1)/2$ multiplications, if the
factored form is used.

Most of the special invariants presented here follow the same pattern and
are derived from the same source, namely from the different action
of $G$ and $H$ on natural objects like the action on blocks or 
block systems.  Ultimately, as we saw above in the discussion of 
general factorisation patterns, we can use permutation
presentations for $G$ acting on the cosets by any subgroup $V<G$.

In the following we assume that $H < G \leq \Sym n$ where $H$ is maximal in $G$ are acting on
$\Omega:=\{X_1,\ldots,X_n\}$. Let us start with the case that $H$ is acting
intransitively. The proof of the following lemma is trivial.
\begin{lemma}
  Assume that there exists an orbit $\O$ of $H$ on $\Omega$ which is not invariant
  under $G$. Then
  \begin{equation}
    \label{eq:intrans}
    F(X_1,\ldots,X_n):=\sum_{X_i\in\O} X_i
  \end{equation}
  is a $G$-relative $H$-invariant.
\end{lemma}
We remark that intransitive groups may occur in our applications even
if we start with transitive groups. The reason is that some of the 
following algorithms will reduce the problem recursively to groups
of smaller degree.

Let us assume for the rest of the section that the given groups
$H\leq G\leq \Sym n$ are transitive. For transitive groups the notion
of blocks and block systems are very important.
We remark that most of 
the following invariants are well known, e.g. see \cite{Gei2,GeKl}.
\begin{definition}
  Let $G\leq \Sym n$ be transitive and $\emptyset\ne B\subseteq \Omega$ be
  a subset. Then $B$ is called a block, if for all $g\in G$ we have
  $B^g \cap B :=\{X^g\mid X\in B\} \cap B \in \{\emptyset,B\}$. 
  Blocks of size $1$ and $n$ are called trivial blocks.
\end{definition}
It is very easy to see that $B^g$ is a block if $B$ is a block. By acting on a block
$B$ we get a partition of $\Omega$ which is called block system.
Therefore
every block is contained in a block system. Furthermore it is easy
to see that the blocks containing $X_1$ are in 1-1 correspondence to
the groups $G_{X_1} \leq U \leq G$, where $G_{X_1} = \Stab_G \{X_1\}$
is the point stabiliser
of $G$ and $U$ is the stabiliser of the block, i.e.
$U=\Stab_G B = \{g \in G\mid B^g=B\}$. 

If $H\leq G$ then clearly every block(-system) of $G$ is a block(-system)
of $H$. But it may be the case that $H$ possesses more blocks.
\begin{lemma}
  Let $H\leq G\leq \Sym n$ be transitive groups and assume that
  $B_1,\ldots,B_m$ is a block system of $H$, but not one of $G$.
  Then
  \begin{equation}
    \label{eq:block}
    F(X_1,\ldots,X_n):=\prod_{i=1}^m \sum_{X\in B_i} X
  \end{equation}
  is a $G$-relative $H$-invariant.
\end{lemma}
\begin{proof}
  Every $h\in H$ only permutes the factors of $F$ and therefore stabilises
  $F$. Let $g\in G\setminus H$. Then there exist $X_i$ and $X_j$ lying in
  the same block which are mapped to different blocks. This produces
  a monomial of $F^g$ containing $X_iX_j$ which does not exist in $F$.
  Since cancellations are impossible, we get the desired result.
\end{proof}

Now we can assume that the block systems of $H$ and $G$ coincide.
Now let $B_1$, $\ldots$, $B_m$ be a block system of $H$ (and
$G$). We can define two canonical actions of $G$ and $H$. One is by
simply permuting the blocks which give transitive permutation
representations $\bar G$ and $\bar H$ on $m$ points. 
We get the following exact sequences of groups:
$$1 \rightarrow N_G \rightarrow G \rightarrow \bar{G} \rightarrow 1,\;
1 \rightarrow N_H \rightarrow H \rightarrow \bar{H} \rightarrow 1$$
where $N_G$ (resp. $N_H$) is the kernel of the permutation representation.

In the case that $N_H = N_G$ we can apply the following lemma. We
remark that we always get $N_H=N_G$ if $\bar{H}\ne \bar{G}$. This is true
because $H$ is a maximal subgroup of $G$ by our general assumption.
\begin{lemma}\label{blocksub}
  Let $H\leq G\leq \Sym n$ be transitive groups with a common block
  system $B_1,\ldots,B_m$. Assume that the above defined normal
  subgroups $N_H$ and $N_G$ are equal.  Let $E(X_1,\ldots,X_m)$ be a $\bar{G}$-relative
  $\bar{H}$-invariant.  Then
  \begin{equation}
    \label{eq:blocksub}
    F(X_1,\ldots,X_n):=E(Y_1,\ldots,Y_m) \text{ for }Y_i:=\sum_{X\in B_i} X
  \end{equation}
  is a $G$-relative $H$-invariant.
\end{lemma}
\begin{proof}
  Elements of $N_H=N_G$ only change the ordering of the sum defining
  $Y_i$. Therefore an element $g$ acts on $F$ via the action of $\bar{g}$
  on $E$. Therefore the polynomial $F$ is $H$-invariant. In order to
  show the $G$-relativity, we need to prove that for $g\in G\setminus H$ we
  have $\bar{g}\notin \bar{H}$. The last statement easily follows
  from $N_H=N_G$.
\end{proof}

The other action can be defined within a block $B_1$ via 
$\Stab_G(B_1)|_{B_1}$. We get the following invariant.
\begin{lemma}\label{blockrel}
  Let $H\leq G\leq \Sym n$ be transitive groups with a common block
  system $B_1$, $\ldots$, $B_m$. Let $\tilde H:=\Stab_H(B_1)|_{B_1}$,
  $\tilde G:=\Stab_G(B_1)|_{B_1}$ and assume $[G:H]=[\tilde G:\tilde
  H]$.  Let $E(X_{i_1},\ldots,X_{i_l})$ where $B_1 = \{X_{i_1},
  \ldots, X_{i_l}\}$ is a $\tilde G$-relative $\tilde H$-invariant.
  Furthermore let $\{\sigma_1,\ldots,\sigma_m\}$ be a
  system of representatives of right cosets of $\Stab_H(B_1)$ in $H$. 

  Then $F:=E^{\sigma_1}+\cdots +E^{\sigma_m}$ is a $G$-relative $H$-invariant.
\end{lemma}
\begin{proof}
  An element of $H$ can be uniquely written as a product of an element
  of $\Stab_H(B_1)$ and some $\sigma_i$. The first one stabilises $E$
  and the second one only permutes the $E^{\sigma_i}$. Therefore $F$
  is invariant under $H$. Since $[G:H]=[\tilde{G}:\tilde{H}]$ we see
  that $\{\sigma_1,\ldots,\sigma_m\}$ are representatives of the right
  cosets of $\Stab_G(B_1)$ in $G$. Since an element $g \in G\setminus
  H$ can be uniquely written as a product $\tilde{g}\sigma_i$ of an
  element $\tilde{g}\in\Stab_G(B_1)$ and some $\sigma_i$ we get that
  the element $\tilde{g}$ cannot be an element of $\Stab_H(B_1)$.
  Therefore $E^{\tilde{g}} \ne E$. Furthermore the $X_j$ which appear
  in $E^{\sigma_i}$ are different for different $i$'s which shows that
  $F^g \ne F$.
\end{proof}

Now we have to deal with groups where the number of block systems is the same
and it is not possible to use Lemma \ref{blocksub} or Lemma \ref{blockrel}.
In this situation, we can try the following (\cite[6.19]{Gei2}) in the situation
that the size of $O$ is not too large:
\begin{lemma}\label{BlockQuotients}
Let $U := \Stab_H(B_1)|_{B_1} = \Stab_G(B_1)|_{B_1}$ and $K_1<K_2\leq U$.
Now let $F$ be a $K_2$-relative $K_1$-invariant such that $O:= F^G = F^H$ and 
the orbit $O$ has the form $\{F^{\sigma_1}, \ldots, F^{\sigma_o}\}$ for suitable elements
$\sigma_1,\ldots,\sigma_o\in H$.
Finally let $\rho : G \to \Sym{O}$ be the permutation representation
of $G$ on $O$. If $\rho(H) \ne \rho(G)$ then let $Y$ be a $\rho(G)$-relative
$\rho(H)$-invariant. For a suitable Tschirnhausen transformation
$t\in \Z[x]$ we have that 
$$I := Y(t(F^{\sigma_1}(\underline X)), \ldots, t(F^{\sigma_o}(\underline X)))$$
is a $G$-relative $H$-invariant.
\end{lemma}
\begin{proof}
  Since $G$ and $H$ act identically on the block $B_1$, the orbits $F^G$
  and $F^H$ are the same.
  By construction, $I$ is clearly $H$-invariant, all that we need to show is 
  that $I$ is not $G$-invariant. Since $F$ is not $\rho(G)$ invariant,
  this is immediate.
\end{proof}
It should be noted that the use of blocks above is only part of the attempt to
create an invariant $F$ with a small orbit.

The following theorem is a generalisation of a result of Eichenlaub \cite{Eic},
who proved the corresponding result for wreath products of symmetric groups.
Recall that a wreath product $U\wr V$ is a semidirect product of
the type $U^m \rtimes V$, where $V\leq S_m$ and the action of $V$ permutes
the copies of $U$. For a formal definition we refer the reader to 
\cite[p. 46]{DiMo}.
 
\begin{theorem}
  Let $G = U \wr V$ be the wreath product acting on $X_{i,j}\;(1\leq i
  \leq d, 1\leq j\leq m)$, where $U\leq S_d$, $V\leq S_m$ and
$md=n$. Furthermore let $N\unlhd U$ be a normal subgroup of index 2.
Let $E$ be a $U$-relative $N$-invariant with the property
that $E^u=-E$ for all $u\in U\setminus N$. Denote by $s_k$ the $k$-th elementary
symmetric function on $m$ letters. Then $G$ has a subgroup $H$ of index 2
and 
$$F(X_{1,1},\ldots,X_{d,m}):=s_m(d_1,\ldots,d_m)=d_1\cdots d_m$$
is a  $G$-relative $H$-invariant,
where $d_j:=E(X_{1,j},\ldots,X_{d,j})$.
\end{theorem}
We remark that in the original statement given in \cite{Eic} there are
two other subgroups of index 2. One is $S_d \wr A_m\leq S_d \wr S_m$
which can be dealt with Lemma \ref{blocksub} and the other one comes
from the fact that whenever we have two subgroups of index 2, there
will be a third one. An invariant for this can be efficiently computed
using the first two invariants, see Lemma \ref{combine}.
\begin{proof}
  Clearly, we have $N\wr V \leq U \wr V$ and using Lemma
  \ref{blockrel} we get that $E+E^u$ with $u\in U\setminus N$ is a
  $G$-relative $N\wr V$-invariant.  Let $u\in U\setminus N$
  be an arbitrary element and let $u_1$ and $u_2$ be the canonical
  images of $u$ in the first and second copy of $U^m$ in $G$,
  respectively. Now we claim that $H = \langle N\wr V, u_1u_2 \rangle\leq G$.
  Clearly, $F$ fixes all elements of $N\wr V$ because all $d_i$ are fixed
  by elements of $N$ and swapped by elements of $V$. The element $u_1u_2$
  fixes $d_3,\ldots,d_m$ and has the property that $d_1^{u_1u_2}=-d_1$ and
  $d_2^{u_1u_2}=-d_2$. Therefore we get $F^{u_1u_2}=F$. For an arbitrary element
  $g\in G$ we get that $F^g=\pm F$ and therefore the index of $H$ in $G$ is
  at most 2.
  Clearly, $F^{u_1}=-F$ and therefore $H\ne G$ and $F$ is $G$-relative.
\end{proof}
We remark that this invariant can be applied to groups $G$ which are not 
wreath products. E.g. it could be possible that $G$ is contained in a wreath
product, but not in the index 2-subgroup and $H$ is contained in that index
2-subgroup.

As already mentioned it is possible to combine relative invariants in order to
get new ones. Suppose $G\leq \Sym n$ has two subgroups $H_1<G$ and $H_2<G$ with
$G$-relative $H_i$-invariants $F_i$. On the invariant 
field side, this corresponds to $\Q(\underline X)^G$
having two finite separable extensions
$\Q(\underline X)^G(F_i)$ corresponding to $\Q(\underline X)^{H_i}$ 
with normal closures
$M^{C_i}$ and $C_i := \Core_G(H_i)$. In this situation we can transfer 
information about $H_i$ to all subfields (and the corresponding fix groups)
of the compositum $M^{C_1}M^{C_2} = M^{\Core_G(H_1\cap H_2)}$.

The first such example already appears in \cite{Eic}.
\begin{lemma} \label{combine}
  Let $G\leq \Sym n$ be a permutation group which has two subgroups
  $H_1\ne H_2$ of index
  2 with $G$-relative $H_i$-invariants $F_i$.
  If $F_i^g = \pm F_i$ for $g\in G$, then $F_1F_2$ is a $G$-relative $H_3$-invariant for
  $H_3:=(H_1 \cap H_2) \cup ((G\setminus H_1) \cap (G\setminus H_2))$.
\end{lemma}
\begin{proof}
  An element of $H_1\cap H_2$ clearly stabilises $F$. Therefore let 
  $h\in H_3\setminus H_1\cap H_2$. Then $h\notin H_1\cup H_2$ and therefore
  $F_1^h=-F_1$ and $F_2^h=-F_2$ which gives $F^h=F$. This proves that $F$ is
  $H_3$-invariant. Let $g\in G \setminus H_3$. Then $g\in H_1$ or $g\in H_2$,
  but $g\notin H_1\cap H_2$. Therefore $F^g=-F$ and $F$ is $G$-relative.
\end{proof}
Even if the invariants do not satisfy $F_i^g=\pm F_i$, the above Lemma 
\ref{combine} can be used, since for $G//H_i = \{\Id, g\}$
we see that $\tilde F_i := F_i -F_i^g$ is a $G$-relative $H_i$-invariant with
the desired property $\tilde F_i^g = -\tilde F_i$.

In the more general situation we can still use the field theoretic
view to combine information from two (or more) subgroups:
Assume $G<\Sym n$ has two subgroups $H_i<G$  ($i=1,2$) with
  $G$-relative $H_i$-invariants $F_i$, set $H_{12} := H_1 \cap H_2$ and
  $C_i := \Core_G(H_i)$ for $i\in\{1,2,12\}$.
  Then for any maximal subgroup $C_{12}<H_3<G$ a $G$-relative $H_3$
  invariant can be constructed by any of the following methods from
  $F_i$, ($i=1,2$) and a $G/C_{12}$-relative $H_3/C_{12}$-invariant.
  Also, set $K := \Q(\underline X)^G$.
  \begin{enumerate}
  \item (Intransitive construction)
  Let $\tilde H_i$ be the permutation representation of $G$ on $G//H_i$,
  $i\in \{1,2,12\}$.
  We consider the subdirect product $\tilde H_1 \times_{H_{12}} \tilde H_2
  \cong \tilde H_{12} \cong G/C_{12}$.
  The maximal subgroup $H_3<G$ corresponds to a maximal subgroup
  $\tilde H_3 < \tilde H_{12}$. Let $F$ be an invariant for this pair, 
  then $F([F_1^s : s \in G//H_1], [F_2^s : s \in G//H_2])$ is
  a $G$-relative $H_3$-invariant.
  \item (Transitive construction) 
    By the primitive element theorem, we can find an invariant $F_i$, $i=1,2$ 
    such that $K(F_i) = \Q(\underline X)^{C_i}$. Again,
    by the primitive element theorem, we find some $r$ such that
    $F_1+rF_2$ is primitive for $\Q(\underline X)^{C_{12}}$. From here
    it is straight forward to obtain an invariant for $H_3$ as a polynomial
    in $F_1+rF_2$.
  \end{enumerate}
In general, this is only applicable if the indices of the 
groups in question are small. In particular, the transitive 
construction is mainly of interest for normal subgroups.

\section{Intransitive groups}
\label{sec:intrans}

The Stauduhar algorithm works for intransitive groups (from reducible
polynomials) in the same way as it does for transitive groups (from
irreducible polynomials).  Let $f\in \Q[x]$ be a squarefree polynomial
of degree $n$. Assume that $f=f_1\cdots f_r\in \Q[x]$ has $r$ factors
of degree $n_i = \deg f_i$. Then we know that $\Gal(f)$ is a
subgroup of the intransitive group $\Sym{n_1} \times \ldots \times
\Sym{n_r}\leq \Sym{n}$.  Using the methods for irreducible polynomials
we can compute the Galois groups $G_i:=\Gal(f_i) \leq \Sym{n_i}$. Then
$\Gal(f) \leq G_1 \times \ldots \times G_r \leq \Sym{n}$. This direct
product can be used as a starting group of our algorithm.

In order to simplify the presentation, let us assume that $\Gal(f)
\leq G_1 \times G_2 < \Sym{n}$. This is no restriction, since we do
not assume that $G_1$ or $G_2$ are transitive. Therefore we have a
corresponding factorisation $f=f_1f_2\in \Q[x]$, where we do not assume
that $f_1,f_2$ are irreducible.  All groups $H$ with $\Gal(f)\leq H
\leq G_1\times G_2$ have a special structure. Let us start to theoretically describe $\Gal(f)$.
Denote by $N_i$ the splitting field of $f_i$. Furthermore define $N$ to be the compositum $N_1N_2$
and $M:=N_1 \cap N_2$. Let $U$ be the Galois group of $M/K$. Then the Galois group of $N/M$ is
the subdirect product (fibre product) $G_1 \times_U G_2$ with common factor group $U$. Denote by
$\phi_i: G_i \rightarrow U$ the corresponding epimorphisms. Then $G_1 \times_U G_2$ can be realized
via 
$$G_1 \times_U G_2 = \{(g_1,g_2) \in G_1\times G_2 \mid \phi_1(g_1) = \phi_2(g_2)\}.$$

Now let us consider the case that $H=G_1 \times_U G_2$ and
$G=G_1\times G_2$.  We remark \cite[Corollary 1.3]{BrMi} that $H$ is a
normal subgroup of $G$, if and only if $U$ is abelian.  Define
$V_i\leq G_i$ to be the normal subgroups such that $G_i/V_i = U$. Then
we get the following chain of subgroups:
$$V_1\times V_2 \leq G_1 \times_U G_2 \leq G_1 \times G_2.$$
A $G_1\times G_2$-relative $V_1 \times V_2$-invariant can be computed by using the corresponding
$G_i$-relative $V_i$-invariants defined on the components and the primitive field argument.
This invariant can be improved to a $G$-relative $H$-invariant by taking sums over elements from
$H//(V_1\times V_2)$.

In general, since the generic invariants are computationally bad, we would
like to use special invariants in this case as well. However, none of them
work for intransitive groups, so our only chance here is to compute
a transitive representation of the larger group and then test for special 
invariants in the transitive representation.
Let $H<G$ and $G$ be intransitive. If the $G$-orbits and the $H$-orbits
differ, we get a trivial $G$-relative $H$-invariant from any $H$-orbit that 
is no $G$-orbit. Hence, we assume that the orbits are the same. 
Similarly, we assume that the  action of $G$ and $H$ on the orbits agree.
In this case we construct a transitive representation $\phi:G \to \Sym T$ 
of $G$ on the
set $T := \prod_{o\in O} o$ where the product runs over all orbits. The image
$\phi(H)$ of $H$ under this representation is again a subgroup of $\phi(G)$ and
we can now test for special invariants:
\begin{lemma} Assume $I\in \Z[X_t\mid t\in T]$ is a $\phi(G)$-relative $\phi(H)$-invariant.
Then there exist a suitable Tschirnhausen transformation $y\in \Z[x]$ such that
$$I(y(\sum_{o\in O} X_{t_o})\mid t\in T)$$
is a $G$-relative $H$-invariant.
\end{lemma}
\begin{proof}
For any fixed $t\in T$ it is clear that 
$\sum_{o\in O} X_{t_o}$ is a primitive element
for $\Q[X_t\mid t\in T]^{\phi(H)}/\Q[X_t\mid t\in T]^{\phi(G)}$ since
all the conjugates $\sum_{o\in O} X_{s_o}$ $s\in T$ are different.
\end{proof}

\section{Computation of Galois Groups}\label{sec7}
In the previous sections we investigated a variety of special and generic
constructions for invariants. Here we are going to discuss how they can be
used to compute Galois group.

To start, let $f$ be an irreducible monic polynomial in $\Z[x]$ and let $K$
be an extension of $\Q$ such that $f(x) = \prod_{i=1}^n (x-\alpha_i)$ with $\alpha_i\in K$,
i.e. $K$ a fixed splitting field, not necessarily a minimal one.
We want to compute
$$\Gal(f) := \Aut(\Q(\alpha_1, \ldots, \alpha_n)/\Q) \leq \Sym{\alpha_1, \ldots, \alpha_n}.$$
Since we assume $f$ to be irreducible, $\Gal(f)$ is a transitive subgroup of $\Sym n$.

Suitable choices for $K$ are $p$-adic fields or the field of complex numbers.
We will defer the choice of the field until we discussed the operations
we need to perform with it. Thus we assume that (somehow) we are given 
a field $K$ and the roots $\alpha_i$ in some arbitrary but fixed ordering.

The main algorithm will, starting with some group $G \geq \Gal(f)$ refine
the initial guess by considering (maximal) subgroups. Hence the first
step is a good starting group.

\subsection{Starting Group}\label{step2}
Naively, obviously, $\Gal(f)\le \Sym n $, so $G := \Sym n$ is a valid 
start. However, this is very bad for the subsequent steps as $\Sym n$ has
maximal subgroups of very large index.

Set $E := \Q(\alpha_1) = \Q(x)/f$ which is a number field of degree $n$.
Using algorithms developed by Kl\"uners \cite{diss_JK} (or recently
van Hoeij, Kl\"uners and Novocin \cite{issac2011}) it is relatively
easy to find subfields $\Q\subset F\subset E$ - or to decide that there are
no subfields. By Galois theory, the subfields are in 1-1 correspondence
to the (unknown) block systems of the (unknown) Galois group.
Thus the non-existence of subfields proves the group to be primitive.

Assume we have a non-trivial subfield $F = \Q(\beta)$
for $\beta = \sum r_\ell\alpha_1^\ell$. Then $B_1 := \{\alpha_j \mid
  \sum r_\ell\alpha_i^\ell = \sum r_\ell \alpha_1^\ell\}$ is the block
containing $\alpha_1$, the other blocks are computed similarly.
From here, it is trivial to compute the wreath product $W_F$ corresponding 
to this block system, and $\Gal(f) \subseteq W_F$. The construction
can be improved if we compute the Galois group of $F/\Q$ (and even more if we 
compute the group of $F/E$ - but this is too expensive in practise).
Doing this for all subfields, we compute a suitable starting group.

If there are no subfields, hence $\Gal(f)$ is primitive, we can try
to obtain a good starting group from factoring suitable
resolvent polynomials as indicated in Theorem \ref{genStaud} and Examples \ref{largefactor} and \ref{largefactor2}.

\subsection{Stauduhar}\label{step3}
We assume now that we have some $\Gal(f)\leq G\leq \Sym n$. 
The next step is now to either prove that $G$ is the Galois group
or replace it by some smaller group $H$.
Now let $H<G$ be maximal. If we have subfields, we should also
verify that $H$ admits the same block systems as $G$, otherwise it cannot
contain the Galois group. In the following we would like to apply Corollary \ref{StaudLin}.

We now find a $G$-relative $H$-invariant $F$ using any of the methods
above, typically starting with the special invariants in Section \ref{special} and, failig that, using \ref{generic}. Next, we verify that 
\begin{equation}
\label{sqf}
F^\sigma(\underline\alpha) = F^\tau(\underline \alpha)\text{ if and only if } \sigma\tau^{-1}\in H
\end{equation}
and compute
$$C := \{ \sigma\in G//H \mid F^\sigma (\underline\alpha)\in \Z\}.$$
If $C$ is non-empty, then $G := \cap_{\sigma\in C} H^\sigma$ will be our
new group.

\begin{remark}\label{short-coset}
Obviously, if $(G:H)$ is large, this is going to be very inefficient, if not
impossible. If we have knowledge of some non-trivial element $\tau\in \Gal(f)$,
coming from some known automorphism of $K$, then we can aid the computation of 
$C$. Instead of $G//H$ we only compute
$$G//_\tau H := \{\sigma\in G//H \mid \tau \in H^\sigma\}$$
the so called short-coset. The actual computation of $G//_\tau H$ can
be performed even if $(G:H)$ is too large to be computed \cite[4.6]{GeKl}.
\end{remark}

However, we do not know how to test \eqref{sqf} effectively. All we do here
is to apply some probabilistic test, i.e. test for some 100 cosets if the
images differ and rely on an independent proof later to justify
the result.

\subsection{Splitting Field}
Now that we have looked at the components of the algorithm, we can discuss
the splitting field. As we saw, we need to be able to quickly evaluate 
$F(\underline \alpha)$, decide if two such evaluations are different and,
finally, test if $F(\underline \alpha)\in \Z$.
All of those tasks would be trivial if we could use an purely algebraic, exact
representation of a splitting field $K$. However, since $[K:\Q] \geq \#\Gal(f)$
this is in general not practical. Using $K=\C$ as Stauduhar did is possible, but
makes it difficult to decide if $F(\underline\alpha)\in \Z$, this would involve
a careful analysis of the numerical properties of $F$.
By restricting the invariants to be free of division, and using a suitable
$p$-adic field $K$, we can overcome most problems, although we actually
need both complex and $p$-adic information. We choose a suitable prime $p$ and
compute a finite extension $K$ of $\Q_p$. The complex information is
used to derive the $p$-adic precision necessary to guarantee correctness.
Let $0<M$ be such that $|\alpha_i|\leq M$ for all complex roots $\alpha_i$. 
It is now easy to compute
$N$ such that $|F^\sigma(\underline\alpha)|\leq N$ - for all $\sigma$ as we
cannot align the ordering of complex and $p$-adic roots.
Thus using a $p$-adic precision $k$ such that $p^k>2N$ means that we
can easily find (the unique) $\theta\in \Z$ such that 
$F(\underline\alpha) = \theta\bmod p^k$ and $|\theta|<N$.

Proving that $F(\underline\alpha) = \theta$ is equivalent to showing that 
$R_F(\theta) = 0$. Since $R_F\in \Z[t]$ and $\theta\in \Z$, we have $R_F(\theta)\in \Z$.
From $F(\underline\alpha) = \theta\bmod p^k$ we get $p^k | R_F(\theta)$, while
on the other hand $|R_F(\theta)|\le (|\theta|+N)^{(G:H)}$, so either
$R_F(\theta)=0$ or
\begin{equation}\label{eq:Prec}p^k\le (|\theta|+N)^{(G:H)}.\end{equation}
so we can easily compute
$k$ large enough to prove $R_F(\theta)=0$.
Unfortunately, $k = O(G:H)$, so $k$ is too large to be useful in general.
Similarly to the use of short cosets (see Remark \ref{short-coset}) we apply a hybrid approach.
We choose $k$ large enough to {\em find} $\theta$, i.e. $k = O(\log N)$
and rely on a final proof step to verify the computation.

\subsection{Checking unproven steps}\label{check}
In the case that $(G:H)$ is huge we have two problems when we apply
Corollary \ref{StaudLin}.  On the one hand the set $G//H$ of coset
representatives of $G/H$ is too big and on the other hand the needed
$p$-adic precision depends exponentially on the index $(G:H)$, see
\eqref{eq:Prec}.  In our actual implementation we only consider the
short cosets in Remark \ref{short-coset} and we replace the exponent
in \eqref{eq:Prec} by a small number like 10. Using these two
modifications we are able to do the corresponding computations for the
Stauduhar step.  In order to get a mathematical proof for our
computations we have two problems.  Firstly, we apply Corollary
\ref{StaudLin}, but we cannot check if the resolvent polynomial $R_F$
is squarefree by only considering some of its roots. Secondly, by only
using exponent 10 instead of $(G:H)$ we cannot prove, that $\theta$ is a
rational integer. In both cases the probability that we are wrong is small,
but this does not give a mathematical justification that $\Gal(f)$ is contained in a conjugate
of $H$. In order to get this we change the method and use Theorem \ref{genStaud} for larger
degree factors $A$ (and a different $H$).

For simplicity we assume we have a subgroup chain $G =: H_0 > H_1 > \ldots > H_r$
as a result of the algorithm with only one (the first) step being unproven. 
We expect that $\Gal(f)= H_r$, but strictly,
at this point we only have the following two facts:
$\Gal(f)\subseteq G = H_0$ and if $\Gal(f)\subseteq H_1$, then $\Gal(f) = H_r$.
Typically, this is the result of $(H_0:H_1)$ being too large to 
verify the resolvent to be squarefree or the derived precision being too
large to verify the rationality of the root.
The correctness of the other steps is depending on the correctness
of the first step. The case of several unproven steps, due to several large indices
can be handled analogously.

We proceed as follows. We try to find a subgroup $U$ of $G$ such that $H_r$ acts intransitively
on the cosets of $G/U$. Then we compute the resolvent polynomial $R_F$ for the group pair $U\leq G$
and some $G$-relative $U$-invariant $F$. If $R_F$ is squarefree and reducible then this already
proves that $\Gal(f)$ is a proper subgroup of $G$. Since we expect that the Galois group is $H_r$
we can easily compute the elements $\sigma_1,\ldots,\sigma_m\in G$ which give the factor
$A$ in Theorem \ref{genStaud}. Therefore we can easily compute an approximation of the expected factor, hence, by rounding, the expected factor in $\Z$.
Finally, we use exact trial division and then descent to a smaller subgroup of $G$ by using this theorem.
When the stabilizer is $H_i$ for some $1\leq i\leq r$, 
our computation is finished successfully, otherwise we replace $G$
by the stabilizer and restart this step. 

It might be difficult to find a good subgroup $U$. Good candidates are
intransitive subgroups of $G$, see \cite[Section 5]{GeKl}. In Example
\ref{largefactor2} we used a polynomial acting on 4-sets. These
$r$-set polynomials have the advantage that we can compute them
quickly symbolically.

\subsection{Overall Algorithm}
To quickly summarise the overall algorithm for irreducible integral polynomials
$f$: we start be factoring $f$ modulo several primes $p$ in order to find
a prime such that he least common multiple of the degrees of the factors
is not too large (and not too small) and that $f$ remains squarefree.
Fixing such a prime to then compute approximations to the roots in the 
$p$-adic field as well as the permutation of the roots corresponding to the
Frobenius of the $p$-adic field. Furthermore, if the Galois group is
$\Sym n$ or $\Alt n$, this too is typically detected just from the 
degrees of the modulo $p$-factors. Finally, approximations to the complex
roots of $f$ are obtained as well.

The next step, as outlined above, is to derive a suitable starting group
for the Stauduhar iteration. Here, we compute all subfields of the
stem field $E=\Q[x]/f$ of $f$, the corresponding block systems and the
largest transitive group $G$ admitting those blocks.

The third step is the iteration of the Stauduhar test \ref{step3} above.
As a result, have a chain of subgroups with the properties described
in \ref{check} above, hence we apply those techniques to verify the result.

\subsection{Reducible Polynomials}
Most of the outlined method applies to reducible polynomials as well, the
key difference is that the groups occurring are naturally intransitive, which
excludes most of the special invariants.
Let $f  = \prod f_i$ be squarefree and monic. We start by fixing a common
splitting field $K$, compute the roots of $f$ and compute the Galois groups
$G_i = \Gal(f_i)\subset \Sym {\underline\alpha}$.
Galois theory now states that $\Gal f < \prod G_i$, so we re-start the
algorithm above with $f$ and $G := \prod G_i$ as a starting group.
We note that only subgroups $H<G$ need to be investigated that project
onto the full Galois groups of the factors, i.e. the final group is
a subdirect product of the $G_i$.

\subsection{Other fields}
Most of the abstract theory described in this paper applies for all infinite 
ground fields as well, hence the algorithm carries over to different
applications. However, there are a few remarks in order: the actual performance
of the overall algorithm depends critically on the splitting field chosen.
In the case of $p$-adic fields, this is a well studied situation with
excellent algorithms already known. In the case of Laurent-series over
finite fields, occuring naturally as splitting fields for polynomials
in $f \in \F_q(t)[x]$, the situation is similar. However, in more general
fields, good descriptions of possible splitting fields are not quite that
easily obtained. In particular, apart from efficiency considerations, 
the test for ``rationality of the resolvent root'' needs to be adapted.
Finally, it should be noted that several of the ``special invariant'' listed
above will not work in small characteristic (in particular in characteristic $2$),
hence the efficiency of the method is endangered.

\section{Numerical Results}
In order to test the procedure outlined in this paper, we applied it
to the complete contents of a database of polynomials \cite{KlMa2} with known 
Galois groups (http://galoisdb.math.upb.de).
This database contains explicit examples, sometimes many,
for most groups of degree $\le 23$. For more than $10^6$ polynomials,
a total of 4835 different Galois groups have been computed. In this range, for 4624
groups the average runtime was less than 5 seconds. Only 5 groups took
more than 30 seconds to compute.

Let us look at an explicit example in detail. Let
$$ f := x^{20} - 308x^{16} + 33396x^{12} - 1554608x^8 + 28579232x^4 - 113379904$$
with Galois group $20T684$ of order 61440. We start by factoring $f$ modulo
several small primes to select $p=89$ for our splitting field which is an
unramified cubic extension of $\Q_{89}$. Next, subfields are computed,
and we recurse by computing the Galois group of the degree 10 subfield
first. Using the subfield data, we conclude that  the Galois group of
$f$ is a subgroup of $20T992$ of order $2^{17}\cdot 3\cdot 5$.
This group has 6 maximal transitive subgroups, of which only one
is a candidate for the target groups, the others can be excluded
by block systems or intersections with other known groups. The only 
group to test further is isomorphic to $20T807$ of index $2^4$, for this pair 
of groups we construct a special invariant using \ref{BlockQuotients}.
The group $20T807$ now has 8 maximal transitive subgroups, 2 of which we
need to test further. For both subgroups, both
isomorphic to $20T684$ of index 2, our algorithm fails to find
special invariants, thus uses the generic ones from section \ref{generic}.
Unfortunately, on evaluation of those invariants, we detect
duplicate values, hence have to resort to Tschirnhaus transformation.
In this example, we end up trying up to 10 different transformations
of degree up to 7 before we find one to remove the duplicate values, hence makes
the resultant squarefree and a descent is found. The resulting 
group again has 4 maximal transitive subgroups, none of which however
are possible, thus the computation terminates.
The ``long'' runtime here is a result of the generic invariants on the
one hand and the need for Tschirnhaus transformations on the other.
By construction, the generic invariants chosen are of minimal degree
but need $>500,000$ multiplications for a single evaluation.
Due partly to the Tschirnhaus transformations, a $p$-adic precision
of $>60$ digits is used which then explains the runtime.

Comparing this to other polynomials with the same group, we see that the
runtime varies substantially (20 - 240 seconds) which is due to the number
of Tschirnhaus transformations used: this depends on the polynomials and
not (directly) on the group. In this example, the ``nice'' structure 
of the polynomial with lots of zero-coefficients indirectly causes the
transformations, while we could ``easily'' fix this by a transformation of the
original polynomial, this would also incur a drastic growth of the 
coefficients, thus rendering this mostly useless.

Overall, the runtime can be seen to depend mainly on the groups as this
determines the invariants and the descent tree transversed. Long runtimes
typically are the result of bad invariants (generic invariants, frequently
if the groups are very similar, i.e. small index). Large index subgroups,
while posing a potential problem for the verification, are frequently
easy to compute with: the short cosets reduce the number
of candidates dramatically and the vastly differing groups make finding
of invariants easy.

\section{Future Work}
The algorithm, as presented here, has two major weaknesses: it needs
to find ``good'' invariants and it ``needs'' a small index in order to 
have verifiable results. Thus more work is needed to increase the number
of ``special'' invariants. In fact, work in this direction has already commenced
e.g. Elsenhans \cite{els} found better invariants for pairs of intransitive
groups and for certain (large) pairs of 2-groups.
In order to address the verification problem,
maybe the use of non-linear factors of the resolvent polynomials
as demonstrated in \ref{largefactor} should be investigated further.

However, as of now, we have a degree independent complete algorithm to
compute Galois groups of univariate polynomials. The algorithm is very
efficient and has been used on polynomials of degree $>100$ already.

\end{document}